\documentclass[a4paper, 10pt, twoside]{amsart}

\usepackage{amsthm}
\usepackage{amsmath}
\usepackage{mathptmx} 
\usepackage{amssymb}
\usepackage{txfonts}
\usepackage{microtype}
\usepackage{enumitem}

\theoremstyle{plain}
\newtheorem{theorem}{Theorem}[section]
\newtheorem{lemma}[theorem]{Lemma}

\theoremstyle{definition}
\newtheorem{definition}[theorem]{Definition}

\theoremstyle{remark}

\begin{document}
\null
\vspace{-1em}

\title{Continuous images of Cantor's ternary set}
\author[F.~Dreher]{F.~Dreher$^{*}$}
\author[T.~Samuel]{T.~Samuel$^{\dagger}$}

\date{\today.\\
\parbox{2.5mm}{$^{*}$}\parbox{32.5mm}{fdreher@uni-bremen.de} - Fachbereich 3 - Mathematik, Universit\"at Bremen, 28359 Bremen, Germany.\\
\parbox{2.5mm}{$^{\dagger}$}\parbox{32.5mm}{tony@math.uni-bremen.de} - Fachbereich 3 - Mathematik, Universit\"at Bremen, 28359 Bremen, Germany.}

\begin{abstract}
The Hausdorff--Alexandroff Theorem states that any compact metric space is the continuous image of Cantor's ternary set $C$. It is well known that there are compact Hausdorff spaces of cardinality equal to that of $C$ that are not continuous images of Cantor's ternary set. On the other hand, every compact countably infinite Hausdorff space is a continuous image of $C$. Here we present a compact countably infinite non-Hausdorff space which is not the continuous image of Cantor's ternary set.
\end{abstract}

\maketitle

\section{Introduction.}

One of the simplest sets that is widely studied by and most important to many mathematicians, in particular analysts and topologists, is Cantor's ternary set (also referred to as the middle third Cantor set), introduced by H.~Smith \cite{HS:1975} and by G.~Cantor \cite{GC:1883}.  Cantor's ternary set is generated by the simple recipe of dividing the unit interval $[0,1]$ into three parts, removing the open middle interval, and then continuing the process so that at each stage, each remaining subinterval is similarly subdivided into three and the middle open interval removed.  Continuing this process \textit{ad infinitum} one obtains a non-empty set consisting of an infinite number of points.
We now formally define Cantor's ternary set in arithmetic terms.

\begin{definition}\label{def:CantorSet}
\textit{Cantor's ternary set} is defined to be the set
\begin{align*}
C \mathrel{:=} \left\{ \sum_{n \in \mathbb{N}} \frac{\omega_{n}}{3^{n}} : \omega_{n} \in \{ 0, 2 \} \; \text{for all} \; n \in \mathbb{N} \right\}.
\end{align*}
\end{definition}

Throughout,  $C$ will denote Cantor's ternary set equipped with the subspace topology induced by the Euclidean metric.  Below we list some of the remarkable properties of Cantor's ternary set.  For a proof of Property (1), more commonly known as the Hausdorff--Alexandroff Theorem, we refer the reader to \cite[Theorem 30.7]{Willard:1970}.  Both F. Hausdorff \cite{Hausdorff:1927} and P. S. Alexandroff \cite{A:1927} published proofs of this result in 1927.  For Properties (2) to (8),  we refer the reader to \cite{Falconer:1990} or \cite[Counterexample 29]{SS:1978}; for a proof of Properties (9) and (10), the definition of the Lebesgue measure, Hausdorff dimension and that of a self-similar set, we refer the reader to \cite{Falconer:1990}.  For basic definitions of topological concepts see \cite{Munkres:2000} or \cite{Willard:1970}.
\begin{enumerate}[label=(\arabic*),leftmargin=0.75cm]
\item Any compact metric space is the continuous image of $C$.
\item The set $C$ is totally disconnected.
\item The set $C$ is perfect.
\item The set $C$ is compact.
\item The set $C$ is nowhere dense in the closed unit interval $[0, 1]$.
\item The set $C$ is Hausdorff.
\item The set $C$ is normal.
\item The cardinality of $C$ is equal to that of the continuum.
\item The one dimensional Lebesgue measure and outer Jordan content of $C$ are both zero.
\item The set $C$ is a self-similar set and has Hausdorff dimension equal to $\log(2)/\log(3)$.
\end{enumerate}

In this note we are interested in whether the Hausdorff--Alexandroff Theorem can be strengthened. To avoid any misunderstandings regarding the compactness condition that might stem from different naming traditions, we shall explicitly define what we mean by compact. Note that a compact space does not have to be Hausdorff.

\begin{definition}[Compact]\label{defn:compact}
Given a subset $A \subseteq X$ of a topological space $(X,\tau)$, an \textit{open cover} of $A$ is a collection of open sets whose union contains $A$.  An \textit{open subcover} is a sub-collection of an open cover whose union still contains $A$.  We call a subset $A$ of $X$ \textit{compact} if every open cover has a finite open subcover.
\end{definition}

It is known that for compact Hausdorff spaces the properties (i) metrizability, (ii) second-countability and (iii) being a continuous image of Cantor's ternary set are equivalent. This follows from the fact that a compact Hausdorff space is metrizable if and only if it is second-countable (see for instance \cite[p. 218]{Munkres:2000}) the Hausdorff--Alexandroff Theorem, and the fact that the continuous image of a compact metric space (in our case the Cantor set $C$) in a compact Hausdorff space is again a compact metrizable space (\cite[Corollary 23.2]{Willard:1970}).

Obviously the cardinality of a space that is the continuous image of $C$ cannot exceed that of the continuum. This restriction on cardinality is a necessary, but not a sufficient condition because there are compact Hausdorff spaces with cardinality equal to that of the continuum which are not second-countable, for instance the Alexandroff one-point compactification of the discrete topological space $(\mathbb{R},\mathcal{P}(\mathbb{R}))$, where $\mathcal{P}(\mathbb{R})$ denotes the power set of $\mathbb{R}$.

Restricting the cardinality even further leads to a sufficient condition.  If we only look at countably infinite target spaces, we can deduce that for compact countably infinite spaces the Hausdorff property already implies metrizability.  In a countably infinite Hausdorff space, every point is a $G_\delta$ point which together with compactness implies that the space is first-countable \cite[Problem 16.A.4]{Willard:1970}; since the space is countably infinite it follows that it is also second-countable. Hence, we have a space which is compact Hausdorff and second-countable and so, using the above mentioned equivalence, metrizable.  Therefore, any compact countably infinite Hausdorff space is the continuous image of $C$.  Is this strong restriction on the space's cardinality a sufficient condition for the Hausdorff--Alexandroff Theorem, i.e. is every compact countably infinite topological space the continuous image of $C$? This is precisely the question we address and, in fact, show that the answer is no by exhibiting a counterexample.

\begin{theorem}\label{thm:MainThm}
There exists a compact countably infinite topological space $( T,\tau )$ which is not the continuous image of $C$.
\end{theorem}

In order to prove this result we require an auxiliary result, Lemma \ref{lem:CounterEx}.  In the proof of this result we rectify an error in \cite[Counterexample 99]{SS:1978}. 

The main idea behind the proof of Theorem \ref{thm:MainThm}  is to choose a specific non-Hausdorff space and to show that if there exists a continuous map from the Cantor set into this space, then the continuous map must push-forward the Hausdorff property of Cantor's ternary set, which will be a contradiction to how the target space was originally chosen.

\section{Proof of Theorem \ref{thm:MainThm}.}\label{sec:proof}

\begin{lemma}\label{lem:CounterEx}
There exists a countable topological space $(T, \tau)$ with the following properties:
\begin{enumerate}[label=(\alph{*}),leftmargin=0.75cm]
\item $(T, \tau)$ is compact,
\item $(T, \tau)$ is non-Hausdorff, and
\item every compact subset of $T$ is closed with respect to $\tau$.
\end{enumerate}
\end{lemma}

\begin{proof}
This proof is based on \cite[Counterexample 99]{SS:1978}.  We define $T \mathrel{:=} (\mathbb{N}\times\mathbb{N}) \cup \left\{x,y\right\}$, namely the Cartesian product of the set of natural numbers $\mathbb{N}$ with itself unioned with two distinct arbitrary points $x$ and $y$.  We equip the set $T$ with the topology $\tau$ whose base consists of all sets of the form:
\begin{enumerate}[label=(\roman{*}),leftmargin=0.75cm]
\item $\{(m,n)\}$, where $(m,n)\in \mathbb{N}\times\mathbb{N}$,
\item $T \setminus A$, where $A \subset (\mathbb{N}\times\mathbb{N}) \cup \{y\}$ contains $y$ and is such that the cardinality of the set $A \cap \{ (m, n) : n \in \mathbb{N}\}$ is finite for all $m \in \mathbb{N}$; that is, the set $A$ contains at most finitely many points on each row (these sets are the open neighbourhoods of $x$) and
\item $T \setminus B$, where $B \subset (\mathbb{N}\times\mathbb{N}) \cup \{x\}$ contains $x$ and is such that there exists an $M \in \mathbb{N}$, so that if $(m, n) \in B \cap (\mathbb{N} \times \mathbb{N})$, then $m \leq M$; that is $B$ contains only points from at most finitely many rows (these sets are the open neighbourhoods of $y$).
\end{enumerate}

Property (a) follows from the observation that any open cover of $T$ contains at least one open neighbourhood $U \supseteq T\setminus A$ of $x$ and one open neighbourhood $V \supseteq T \setminus B$ of $y$ with $A$ and $B$ as given above. The points not already contained in these two open sets are contained in $T \setminus (U \cup V) \subseteq T \setminus \left(\left(T\setminus A\right)\cup \left(T\setminus B\right)\right) = A \cap B$ which, by construction, is a finite set. In this way a finite open subcover can be chosen and hence the topological space $(T, \tau)$ is compact.

To see why $(T, \tau)$ has Property (b), consider open neighbourhoods of $x$ and $y$. An open neighbourhood of $x$ contains countably infinitely many points on each row of the lattice $\mathbb{N}\times\mathbb{N}$; an open neighbourhood of $y$ contains countably infinitely many full rows. It follows that there are no disjoint open neighbourhoods $U \ni x$ and $V \ni y$ and thus $T$ is non-Hausdorff.

We use contraposition to prove Property (c). Suppose that $E \subset T$ is not closed. Note that we may assume that $E$ is a strict subset of $T$ since $T$ itself is closed by the fact that $\emptyset \in \tau$.  By construction of the topology, a set that is not closed cannot contain both $x$ and $y$. Also, there needs to be at least one point in the closure $\overline{E}$ of $E$, but not already in $E$; this has to be one of the points $x$ or $y$, because singletons $\{(m,n)\}$ which are subsets of the lattice $\mathbb{N}\times\mathbb{N}$ are open.  Thus the point $(m,n)$ cannot be a limit point of $E$. We shall now check both cases, that is, (i) if $x\in \overline{E}\setminus E$ and (ii) if $y\in \overline{E}\setminus E$.
\begin{enumerate}[label=(\roman{*}),leftmargin=0.75cm]
\item If $x\in \overline{E}\setminus E$, then every open neighbourhood of $x$ has a non-empty intersection with $E$. It follows that there is at least one row in $\mathbb{N}\times\mathbb{N}$ that shares infinitely many points with $E$. Denote this row by $B$.  Then the open cover $\{T \setminus B\} \cup \{ \{b\} : b \in B \cap E \}$ of $E$ cannot be reduced to a finite open subcover, and therefore, $E$ is not compact.
\item If $y\in \overline{E}\setminus E$, then, similar to (i), we have that $E$ contains points from infinitely many rows. Take one point from each of these rows and call the resulting set $A$. Then the open cover $\{T \setminus A\} \cup \{ \{ a \} : a \in A \cap E \}$ of $E$ cannot be reduced to a finite open subcover and hence $E$ is not compact.
\end{enumerate}
\vspace{-1.5em}
\end{proof}

\begin{proof}[Proof of \ref{thm:MainThm}]
Assume that there exists a surjective continuous map $f: C \to T$.  We will show that this implies that $\left(T,\tau\right)$ is Hausdorff which contradicts Lemma \ref{lem:CounterEx}(b).

Choose two distinct points $u, v\in T$. Since singletons are compact, Lemma \ref{lem:CounterEx}(c) implies that the sets $\left\{ u \right\}$  and $\left\{ v \right\} $ are closed in $T$ with respect to $\tau$. Therefore, their pre-images under $f$ are non-empty closed subsets of $C$ and have disjoint open neighbourhoods $U(u)$ and $U(v)$, as $C$ is normal.  The complements $U(u)^{c}$ and $U(v)^{c}$ of these open neighbourhoods are compact subsets of $C$.  Thus, their images under $f$ are compact in $T$ because $f$ is continuous and they are closed because of Lemma \ref{lem:CounterEx}(c). Therefore, $V\left( u \right)\mathrel{:=}f\left(U\left( u \right)^{c}\right)^{c}$ and $V\left( v \right)\mathrel{:=}f\left(U\left( v \right)^{c}\right)^{c}$ are open neighbourhoods of $u$ and $v$ respectively. We claim that these sets are disjoint:
\begin{align*}
V\left( u \right)\cap V\left(v \right) = f\left(U\left(u \right)^{c}\right)^{c}\cap f\left(U\left(v \right)^{c}\right)^{c} &= \left(f\left(U\left( u \right)^{c}\right)\cup f\left(U\left( v \right)^{c}\right)\right)^{c}\\
 &=\left(f\left(U\left(u \right)^{c}\cup U\left( v \right)^{c}\right)\right)^{c}\\
 &=\left(f\left(\left(U\left( u \right)\cap U\left( v \right)\right)^{c}\right)\right)^{c}
 =\left(f\left(\emptyset^{c}\right)\right)^{c}
 = \emptyset.
\end{align*}
Hence we have separated the points $u$ and $v$ by open neighbourhoods. Since $u, v \in T$ were chosen arbitrarily, we conclude that $\left(T,\tau\right)$ is Hausdorff, giving the desired contradiction.
\end{proof}

\end{document}